\documentclass{amsart}

\usepackage[english]{babel}
\usepackage[T1]{fontenc}
\usepackage[letterpaper,top=2cm,bottom=2cm,left=3cm,right=3cm,marginparwidth=1.75cm]{geometry}

\usepackage{amsmath}
\usepackage{graphicx}
\usepackage[colorlinks=true, allcolors=blue]{hyperref}

\usepackage[T2A]{fontenc}
\usepackage[utf8]{inputenc}
\usepackage{amsthm, amsmath, amsfonts, amssymb, mathtools}

\usepackage{titleps}
\usepackage{soulutf8}
\usepackage{fontsize}
\usepackage{bbm}
\usepackage{cite}

\usepackage[mathscr]{eucal}

\usepackage{mathtools}
\usepackage{enumerate}
\usepackage{accents}
\usepackage{comment}
\usepackage{dsfont}

\usepackage{xcolor}
\usepackage[colorlinks=true]{hyperref}
\hypersetup{urlcolor=blue,citecolor=red,linkcolor=blue}
\usepackage[initials]{amsrefs}
\newtheorem{theorem}{Theorem}[section]
\newtheorem{lemma}{Lemma}[section]
\theoremstyle{remark}

\theoremstyle{problem}
\newtheorem*{problem}{Problem}
\newtheorem*{definition}{Definition}

\newcommand{\N}{\mathbb N}

\newcommand{\sgn}{\mathtt{sgn}}
\begin{document}
\title{Plasticity of the unit sphere of the space $c_0$}
\author{Olesia Zavarzina}
\address{Department of Mathematics and Informatics, V. N. Karazin Kharkiv National University, 61022 Kharkiv, Ukraine.}
\email{olesia.zavarzina@yahoo.com}
\begin{abstract}
In this paper, we prove that the unit sphere of the space $c_{0}$ is plastic, i.e., every non-expansive bijection from this sphere onto itself is an isometry.
\end{abstract}
\subjclass[2020]{46B20, 46C05}
\keywords{Non-expansive map, non-contractive map, unit sphere, isometry, expand-contract plastic space}
\thanks{ The author was supported  by the National Research Foundation of Ukraine, project No. 2025.07/0369 ``Qualitative methods of nonlinear analysis of heterogeneous structures''.}
\maketitle

\section{Introduction}
\begin{definition}
   A metric space $(M, d)$ is called  \emph{ plastic}  if every bijective non-expansive map $F \colon M \to M $ is an isometry.
\end{definition}
Recall that a map $F$ is called non-expansive (non-contractive) if for any $x,y\in M$
$$d(F(x),F(y))\leq d(x,y).$$
A map $F$ is called non-contractive if for any $x,y\in M$ the inverse inequality holds.
In \cite{CKOW},  Cascales, Kadets, Orihuela, and Wingler posed the following question:
\begin{problem}
    Is it true that the unit balls of all Banach spaces are plastic?
\end{problem}
This question has a number of partial positive results (see, for example, \cite{CKOW, KadZav, Leo, HLZ, Fak} and references therein), but is still open in full generality in infinite dimensions. In finite dimensions, the answer is positive due to the result on totally bounded metric spaces obtained by Naimpally, Piotrowski, and Wingler \cite{NPW}.  The problem of plasticity of the unit balls of Banach spaces may be considered as a continuation of the well-known results of Mazur–Ulam \cite{MU} extended by Mankiewicz \cite{Manc}.
\begin{theorem}(Mazur–Ulam-Mankiewicz.)
Every bijective isometry between the unit balls $B_{X}$ and $B_{Y}$ extends to a linear bijective isometry from $X$ onto $Y$. 
\end{theorem}
In 1987, Tingley \cite{Ting} raised the question of whether the unit balls in the aforementioned theorem could be replaced by unit spheres:
\begin{problem} (Tingley). Let $f \colon S_X \to S_Y$ be a bijective isometry between the unit spheres of two Banach spaces $X$ and $Y$. Does $f$ admit an extension to a linear isometry between $X$ and $Y$?
\end{problem}
Despite numerous partial results for specific spaces (see, for example, \cite{KadMar}, \cite{Tan}), Tingley's problem remains a challenging open question in dimensions three and higher. Moreover, it has been fully resolved in the two-dimensional setting by Banakh only in 2022 \cite{Banakh}. As its natural continuation, one may consider the plasticity problem for the unit spheres of Banach spaces, posed by V. Kadets in personal communication, 2024.
\begin{problem}(Kadets.)
    Is it true that the unit spheres of all Banach spaces are plastic?
\end{problem}
In \cite{LevZav} this question was answered positively for Hilbert spaces, $\ell_1$, $\ell_{\infty}$, and $c$. Observe also, for the unit ball of the space $c_0$ only conditional plasticity was shown. Namely, a bijective non-expansive map $F \colon B_{c_0} \to B_{c_0} $ is an isometry if $F^{-1}$ is continuous \cite{Leo}. It is currently unknown whether this additional condition may be omitted. As the title says, the main objective of this paper is to prove the plasticity (without any additional conditions) of $S_{c_0}$. 

To facilitate our analysis, we introduce the necessary notation and structural definitions used throughout the text. For a real Banach space $X$, let $S_X$ and $B_X$ denote its unit sphere and closed unit ball, respectively. 
In the context of the classical sequence space $c_0$, an element $x \in S_{c_0}$ is represented as a sequence $x = (x^1, x^2, \dots)$ such that $\lim_{n\to\infty} x^n = 0$ and $\max_{n\in\mathbb{N}} |x^n| = 1$. 

We denote by $S_{c_0}^*$ the dense subset of $S_{c_0}$ consisting of all elements with finite support. For any vector $x \in c_0$, its support is defined as the set of indices corresponding to non-zero coordinates:
\begin{equation*}
    \operatorname{supp}(x) = \{n \in \mathbb{N} : x^n \neq 0\}.
\end{equation*}
Thus, $x \in S_{c_0}^*$ if and only if $|\operatorname{supp}(x)| < \infty$. 

Furthermore, let $\{e_n\}_{n \in \mathbb{N}}$ be the standard unit vector basis of $c_0$, where $e_n$ has $1$ at the $n$-th position and zeros elsewhere. We define $E$ as the set consisting of all standard basis vectors and their negatives:
\begin{equation*}
    E = \{\pm e_n : n \in \mathbb{N}\}.
\end{equation*}
Evidently, $E \subset S_{c_0}^* \subset S_{c_0}$. For any real number $t \in \mathbb{R}$, its signum function is denoted by $\sgn(t)$.

The paper is organized as follows. In the second section, we collect and modify some results known for the unit spheres of some other sequence spaces or for $B_{c_{0}}$ to adapt them to the case of $S_{c_{0}}$. These results imply that any bijective non-expansive mapping is tightly bound to a coordinate permutation and a sign-switching sequence. Roughly speaking, we ``turn off'' this coordinate permutation and sign-switching and obtain a function $\widetilde{F}$. This section also includes some previously known or slightly modified  properties of this function. Section 3 is devoted to further investigation of the properties of $\widetilde{F}$, which finally allowed us to see that $\widetilde{F}$  reduces to identity, which in turn means that $F$ reduces to an isometry.

\section{Preliminaries}
 The analog of the first lemma of this section is  Lemma 3.1 from \cite{Leo}, which was proved for $B_{c_0}$ and $B_{c}$.   The ``only if'' part was also proved in Lemma  4.1 in \cite{LevZav} for the unit spheres of $\ell_{\infty}$ and $c$.  Although this result cannot be considered as a completely new one, for the reader's convenience, we give it with proof. 
\begin{lemma}\label{linf1}
   The points $p,q \in S_{c_0}$ satisfy the condition that for every $x \in S_{c_0}$
    \begin{equation*}
        (\|x-p\|\leq1)\vee(\|x-q\|\leq1),
    \end{equation*}
    if and only if $p \in E$ and $q=-p$.
\end{lemma}

\begin{proof}
Let us first show the ``if'' part. Suppose that without loss
of generality $p= e_n$, $q=-e_n$ and fix arbitrary $x\in S_{c_0}$. If $x^n=0$, then both inequalities from the lemma statement hold.  If $x^n >0$, then $\|x-p\|\leq1$, otherwise  $\|x-q\|\leq1$.

    Conversely, let us prove the ``only if'' part. Suppose by contradiction that $|\mathtt{supp}(p)| > 1$. Then there exist two distinct indices $k$ and $j$ such that $p^k \neq 0$ and $p^j \neq 0$ and an index $l$, which without loss of generality we may assume different from $k$, such that $q^l \neq 0$.  Let us construct a vector $x \in S_{c_0}$ as follows: set $x^k = -\mathtt{sgn}(p^k)$, and $x^l = -\mathtt{sgn}(q^l)$. All other coordinates of $x$ are set to zero. Then we clearly obtain:
\begin{equation*}
    \|x-p\| \geq |x^k - p^k| = 1 + |p^k| > 1,
\end{equation*}
and by the choice of $x^l$, we also have $\|x-q\| \geq |x^l - q^l| > 1$, which contradicts the assumption. 

Thus, $|\mathtt{supp}(p)|=|\mathtt{supp}(q)| = 1$. Since $p,q \in S_{c_0}$, this implies $p = \pm e_n$ for some $n \in \mathbb{N}$, meaning $p \in E$. By symmetry, we also get $q \in E$, so $q = \pm e_m$. If $m \neq n$, we can easily find $x \in S_{c_0}$ to violate the condition. Finally, if $q = p$, then choosing $x = -p$ yields $\|x-p\| = \|-2p\| = 2 > 1$, which is another contradiction. Therefore, we have $q = -p$.
\end{proof}

The next lemma  was also proved in \cite{LevZav} (Lemma 4.2) for $c$ and $\ell_{\infty}$, however, it also works for $c_0$.

\begin{lemma}\label{linf2}
    Let $F: S_{c_0} \rightarrow S_{c_0}$ be a non-expansive bijection, and let $x \in E$. Then $F(x) \in E$ and $F(-x) = -F(x)$.
\end{lemma}

\begin{proof}
    For every $y \in S_{c_0}$ either $\|y-x\|\leq 1$ or $\|y+x\| \leq 1$, and the non-expansiveness of $F$ implies that $\min(\|F(y)-F(x)\|,\|F(y)-F(-x)\|)\leq 1$ for every $y \in S_{c_0}$. $F$ is bijective, so $\min(\|y-F(x)\|,\|y-F(-x)\|)\leq 1$ for every $y \in S_{c_0}$. Now all that remains is to apply the previous lemma.
\end{proof}

The next lemma was extracted from Lemma 3.2  in \cite{Leo}, where the same was proved for $B_{c_0}$.

\begin{lemma}\label{c0_step0}
    Let $F: S_{c_0}\rightarrow S_{c_0}$ be a non-expansive bijection. Then there is a bijection $\sigma\colon\mathbb{N}\to\mathbb{N}$ and a sequence $\alpha\colon\mathbb{N}\to\{-1,1\}$ such that for every $x\in S_{c_0}$ and $n\in\mathbb{N}$ the following holds:
    \begin{enumerate}[1)]
        \item if $x^n=0$, then $F(x)^{\sigma(n)}=0$;
        \item if $x^n<0$ and $\alpha^n=1$ ($\alpha^n=-1$), then $F(x)^{\sigma(n)}\leq0$ ($F(x)^{\sigma(n)}\geq0$);
        \item if $x^n>0$ and $\alpha^n=1$ ($\alpha^n=-1$), then $F(x)^{\sigma(n)}\geq0$ ($F(x)^{\sigma(n)}\leq0$).
    \end{enumerate}
\end{lemma}
\begin{proof}
    Let us fix an arbitrary $n\in \N$.  By Lemma \ref{linf2}, for the elements $e_n$ and $-e_n$, we have $F(e_n)= e_k$ or $F(e_n)=-e_k$,  and correspondingly, $F(-e_n)= -e_k$ or $F(-e_n)=e_k$ for some $k\in \N$. Define $\sigma(n)=k$ and $\alpha^n=(F(e_n))^k$. The mapping $\sigma$ defined in this way is injective due to the injectivity of $F$.

    In what follows, we show that conditions 1)--3) hold for all $x\in S_{c_0}$ and $n\in\N$.
    Since $S_{c_0}^*$ is dense in $S_{c_0}$ and $F$ is continuous, it suffices to verify these conditions only for $x\in S_{c_0}^*$, that is, for $x\in S_{c_0}$ with $|\mathtt{supp}(x)|<\infty$. 

    For every $n\in \mathtt{supp}(x)$ define $l_n=-\mathtt{sgn}(x^n)e_n$. Observe  that for every $t\in S_{c_0}$, either $\|t-x\|\leq 1$ or there exists $n\in \mathtt{supp}(x)$ such that $\|t-l_n\|\leq 1$. To demonstrate this, let us fix $t\in S_{c_0}$ and suppose $\|t-l_n\|> 1$ for every $n\in \mathtt{supp}(x)$.  By the definition of  $l_n$, this means $\mathtt{sgn}(x^nt^n)>0$ for every $n\in \mathtt{supp}(x)$ and thus, $\|t-x\|\leq 1$. Since $F$ is non-expansive and bijective, for every $t\in S_{c_0}$, either $\|t-F(x)\|\leq 1$ or there is $n\in \mathtt{supp}(x)$ such that $\|t-F(l_n)\|\leq 1$.
    Further, we will construct a special sequence $z\in S_{c_0}$ and show $\|z-F(x)\|\leq 1$. By the construction of $z$, the latter inequality will automatically imply items (2) and (3). Let
    \begin{equation*}
    z^k=\begin{cases} \mathtt{sgn}(x^n)\alpha^n, \text{ for }  k=\sigma(n), n\in \mathtt{supp}(x),\\
    0, \text{ otherwise}.
    \end{cases}
    \end{equation*}
    Observe by its definition, $z\in S_{c_0}$ and $\|z-F(l_n)\|> 1$ for every $n\in \mathtt{supp}(x)$, so  $\|z-F(x)\|\leq 1$, and thus items (2) and (3) are proved.     
    
    

    It remains to prove item (1). Let us fix $k\not\in \sigma(\mathtt{supp}(x))$ and change the definition of $z$ such that $z^k=1$. We still have $z\in S^*_{c_0}$ and  $\|z-F(l_n)\|> 1$ for every $n\in \mathtt{supp}(x)$. Thus, $\|z-F(x)\|\leq 1$, which implies $F(x)^k\geq 0$. On the other hand, setting $z^k=-1$, we have the opposite inequality  $F(x)^k\leq 0$. Thus, item (1) is proved.

    The last step is to prove surjectivity of $\sigma$. Let us argue by contradiction. Suppose there is $j\not\in\sigma(\N)$. From the above we have $F(x)^j=0$ for every $x\in S_{c_0}^*$. Using the continuity of $F$ and density of $S_{c_0}^*$ in $S_{c_0}$ once again, we obtain $F(x)^j=0$ for every $x\in S_{c_0}$. This contradiction to the surjectivity of $F$ completes the proof.
    
    
\end{proof}
The preceding lemma yields a bijection $\sigma \colon \mathbb{N} \to \mathbb{N}$ and a sequence $\alpha \colon \mathbb{N} \to \{-1,1\}$. Let us define the following isometric automorphism  $A$ of $c_0$ by 
$$(A(x))^{\sigma(n)} = \alpha^n x^n,$$
for all $n \in \mathbb{N}$.
Denote the restriction of this automorphism to $S_{c_0}$ by $G$. Observe that  $G$ is a bijective isometry of $S_{c_0}$ onto itself. To prove the plasticity of $S_{c_0}$, it suffices to show $F=G$, or equivalently, to show that $\widetilde{F} := G^{-1} \circ F$ is the identity map on $S_{c_0}$.
We shall establish this fact in what follows.
 By its definition, $\widetilde{F}$ is a non-expansive bijection from $S_{c_0}$ onto itself. Furthermore, Lemma \ref{c0_step0} implies that $\widetilde{F}$ acts as the identity on $E$ and satisfies the following properties (also extracted from analogous results for $B_{c_0}$ proved in \cite{Leo}).
\begin{lemma}\label{c0_step1}
    For each $x\in S_{c_0}$ and $n\in\N$,
    \begin{enumerate}[1)]
        \item if $x^n=0$, then $\widetilde{F}(x)^n=0$;
        \item if $x^n<0$, then $\widetilde{F}(x)^n\leq0$;
        \item if $x^n>0$, then $\widetilde{F}(x)^n\geq0$.
    \end{enumerate}
\end{lemma}
\begin{lemma}\label{c0_step1a}
    For each $y\in S_{c_0}$ and $n\in\N$, if $y^n<0$, then $\widetilde{F}^{-1}(y)^n<0$, and if $y^n>0$, then $\widetilde{F}^{-1}(y)^n>0$.
\end{lemma}

The next lemma is extracted from Lemma 4.7 from \cite{LevZav}. This property was established for $\ell_{\infty}$ and $c$. However, almost the same argument works for $c_0$.
\begin{lemma}\label{linf7} Let
    $x\in S_{c_0}$, $y=\widetilde{F}(x)$. Then for every $k \in \N$ the following inequalities hold:
    \begin{equation}\label{ineq1}
        |y^k|\leq |x^k|,
    \end{equation}
    and
    \begin{equation}\label{ineq2}
        \mathtt{sgn}(x^k)  \mathtt{sgn}(y^k) \geq 0.
    \end{equation}
\end{lemma}

\begin{proof}
    Let $k \in \N$ be such that $|x^k|<1$, i.e. $\mathtt{max}(\|x-e_k\|,\|x+e_k\|)<2$. The non-expansiveness of $\widetilde{F}$ then implies that $\mathtt{max}(\|y-e_k\|,\|y+e_k\|)<2$, i.e. $|y^k|<1$, and 
    \[|x^k|=\|x-e_k\|+\|x+e_k\|-2\geq \|y-e_k\|+\|y+e_k\|-2=|y^k|.\]
    This proves (\ref{ineq1}) for all $k\in \N$ such that $|x^k|<1$, and (\ref{ineq1}) is obviously true for all $k\in \N$ such that $|x^k|=1$. The claim (\ref{ineq2}) follows directly from  Lemma \ref{c0_step1}.
\end{proof}

\section{Main results}
Let us consider the set $K$ of points with finite support, such that for any $x\in K$   and any   $x^i\in \mathtt{supp}(x)$, $|x^i|=1$.
\begin{lemma}\label{c0_identity_K}
The map $\widetilde{F}$ is the identity map on $K$.
\end{lemma}

\begin{proof} We will proceed by induction on the cardinality of the support. The base case of induction holds since $\widetilde{F}$ is the identity on $E$. Now, suppose for any $a\in K$ with $|\mathtt{supp} (a)|\leq n-1$, $\widetilde{F}(a)=a$. Further, let us argue by contradiction. Without loss
of generality we may consider 
$$x=(x^1,x^2,...,x^n,0,...),$$ 
where $|x^i|=1$ for all $i\leq n$,  and suppose $\widetilde{F}(x)\neq x$.
 By the previous results, one has
$$\widetilde{F}(x)=(y^1, y^2,..., y^n,0...),$$
where, at least for some $1\leq i\leq n$, $0<|y^i|<1$ and all $y^i$ have the same signs as the corresponding $x^i$. Consider  the following element
$$\widetilde{x}=(\sgn(x^1)(1-\varepsilon),...,\sgn(x^{i-1})(1-\varepsilon),\sgn(x^i),\sgn(x^{i+1})(1-\varepsilon),...,\sgn(x^n)(1-\varepsilon),0,...),$$ 
with some $\varepsilon>0$. By Lemma \ref{linf7}, $|\widetilde{F}(\widetilde{x})^k| \leq (1-\varepsilon)$, for all $k\neq i$. Moreover, 
$$\|\widetilde{F}(x)-\widetilde{F}(\widetilde{x})\|=\mathtt{max}_{1\leq k\leq n}\{|\widetilde{F}(x)^k-\widetilde{F}(\widetilde{x})^k|\}\leq \|x-\widetilde{x}\|=\varepsilon.$$
Thus, in particular,
$$|\widetilde{F}(x)^i-\widetilde{F}(\widetilde{x})^i|\leq\varepsilon,$$
and since  $\varepsilon$ can be chosen arbitrarily  small, we can achieve $|\widetilde{F}(\widetilde{x})^i|<1$.
We obtain the contradiction $\widetilde{F}(\widetilde{x}) \notin S_{c_0}$, which completes the proof. 

\end{proof}
\begin{lemma}\label{one_coord}
Let $x$ and $y$ differ in one single coordinate. Moreover, $|x^i|>|y^i|$  and $\widetilde{F}(x)=x$. Then $\widetilde{F}(y)^i=y^i$.
\end{lemma}
\begin{proof}
By Lemma \ref{linf7}, we have 
$$|\widetilde{F}(y)^i|\leq|y^i|,$$
and they have the same signs.
On the other hand
$$|x^i-y^i|=\|x-y\|\geq \|\widetilde{F}(x)-\widetilde{F}(y)\|=\|\widetilde{F}(y)-x\|\geq |x^i-\widetilde{F}(y)^i|.$$
Together, these conditions give us the needed conclusion.
\end{proof}
\begin{theorem}
$\widetilde{F}$ is the identity map on $S_{c_0}$.
\end{theorem}
\begin{proof}
Again, since $\widetilde{F}$ is continuous and $S_{c_0}^*$ is dense in $S_{c_0}$, it suffices to  consider only elements  $x\in S_{c_0}^*$, that is, $|\mathtt{supp}(x)|< \infty$. 
Let us show $\widetilde{F}(x)=x$ gradually increasing the number of nonzero coordinates $x^k$. 
If $x$ has only one nonzero coordinate, then $x=\pm e_k$ for some $k\in \N$ and there is nothing to prove. If $x$ has two nonzero coordinates, the case with moduli of these coordinates equal to 1 is covered by Lemma \ref{c0_identity_K}. Now consider x of the following form
$$x=(x^1,x^2,0,\dots),$$
with $|x^1|=1$ and $|x^2|<1$.
  Then $|\widetilde{F}(x)^1|$ is obviously equal to 1. Consider $\widetilde{x}$ with $\widetilde{x}^k=
\mathtt{sgn}x^k$ and use Lemma \ref{one_coord}. We have
$\widetilde{F}(x)^2= x^2$. 

Analogously, by changing only one coordinate to its sign at each step, we may show $\widetilde{F}(x)=x$ for any $x \in S_{c_0}^*$ with arbitrarily many coordinates of modulus 1 and any finite number of coordinates whose moduli are less than 1. 

Thus, $\widetilde{F}(x) = x$ holds for all $x \in S_{c_0}^*$. Finally, by the density of $S_{c_0}^*$ and the continuity of $\widetilde{F}$, we conclude that $\widetilde{F}$ is the identity map on the entire sphere $S_{c_0}$.
\end{proof}

\bibliographystyle{alpha}
\bibliography{sample}

\end{document}